\documentclass[12pt]{amsart}
\usepackage{amsmath}
\usepackage{amssymb}
\usepackage{amscd}
\usepackage[all]{xy}
\usepackage{appendix}
\usepackage{hyperref}
\def\lra{\longrightarrow}
\def\map#1{\,{\buildrel #1 \over \lra}\,}

\def\oo{\otimes}
\newcommand{\pc}{{\mathrm{pc}}}

\def\zar{{\mathrm{zar}}}

\def\cF{\mathcal F}
\def\cI{\mathcal I}

\def\cMpctf{{\mathcal M}_{\text{pctf}}}
\def\Kzar{K^{\mathrm{zar}}}
\def\KQ{K^Q}
\def\Sh{\mathbf{Sh}}
\newcommand{\bbH}{\mathbb H}
\newcommand{\Schk}{\mathrm{Sch}/k}

\newcommand{\A}{\mathbb{A}}
\newcommand{\G}{\mathbb{G}}
\newcommand{\R}{\mathbb{R}}

\newcommand{\Z}{\mathbb{Z}}
\newcommand{\N}{\mathbb{N}}
\newcommand{\bbP}{\mathbb{P}}

\newcommand{\bu}{\mathbf{u}}

\newcommand{\p}{\mathfrak{p}}
\newcommand{\Hom}{\operatorname{Hom}}

\newcommand{\pt}{\mathbf{pt}}
\def\Spec{\operatorname{Spec}}
\def\MSpec{\operatorname{MSpec}}
\def\MProj{\operatorname{MProj}}

\newcommand{\Gm}{\G_{\text m}}

\input xy
\xyoption{all}

\numberwithin{equation}{section}

\theoremstyle{plain}
\newtheorem{thm}[equation]{Theorem}
\newtheorem{prop}[equation]{Proposition}
\newtheorem{lem}[equation]{Lemma}
\newtheorem{cor}[equation]{Corollary}

\theoremstyle{definition}
\newtheorem{defn}[equation]{Definition}
\newtheorem{ex}[equation]{Example}

\theoremstyle{remark}
\newtheorem{rem}[equation]{Remark}

\begin {document}
\title{$K$-theory of Matroids and Monoid Schemes}
\date{\today}
\author{Christian Haesemeyer}
\address{School of Mathematics and Statistics, University of Melbourne,
VIC 3010, Australia}
\email{christian.haesemeyer@unimelb.edu.au}
\urladdr{https://blogs.unimelb.edu.au/christian-haesemeyer/}

\author{Charles Weibel}
\address{Math.\ Dept., Rutgers University, New Brunswick, NJ 08901, USA}
\email{weibel@math.rutgers.edu}\urladdr{http://math.rutgers.edu/~weibel}
\begin{abstract}
This paper continues the study of the $K$-theory of monoid schemes,
using it to give 
a useful definition of the higher $K$-theory of a matroid
via its Bergman fan.
\end{abstract}
\maketitle

We define the higher $K$-theory of a matroid $M$
as the higher $K'$-theory of its associated toric monoid scheme.
This agrees with the existing definition of $K_0(M)$
(see \cite{LLPP}).
In more detail,
a matroid $M=(E,\cF)$ has an associated Bergman fan $\Delta$,
a toric fan built from its lattice of flats $\cF$, and a toric fan
has an associated toric monoid scheme $X$. The realization
$X_k$ of $X$ over a field $k$ is a toric variety.

\begin{defn}\label{def:K-of-matroid}
The $K$-theory spectrum $K(E,\cF)$ of a matroid $M=(E,\cF)$ is defined
to be the spectrum $\Kzar(X) \cong K'(X)$, where $X$ is
the toric monoid scheme associated to the Bergman fan of $M$.
We write $K'_n(M)$ for the homotopy group
$K'_n(X) = \pi_nK'(X)$. 
\end{defn}

Here $K'$ is the $K'$-theory defined in \cite{HW}, and $\Kzar$ is a
Zariski sheafified version of the $K$-theory of \cite{ELY} that we
introduce in \ref{defn:K-Zar} below.
Along the way, and to justify our definition, we prove a
number of foundational results about $\Kzar$ and $K'$ of monoid
schemes. For example, we prove in Theorem \ref{thm:K-of-sm-proj}
that if $X$ is a smooth and proper monoid scheme then
the groups $K'_n(X)$ are finite for $n>0$, and are determined by
$K'_0(X)$ and the stable homotopy groups $\pi_n^s$.

\medskip

Here is how we have organized our exposition.
In Section \ref{sec:toric}, we briefly recall the definition and
terminology of monoid schemes used in this paper, following
\cite{CHWW}. For example, we restrict to monoid schemes which are
partially cancellative and torsion free (pctf).

In section \ref{sec:higher-K}, we review the $K'$-theory of a pctf
monoid scheme $X$, such as the monoid scheme associated to a fan.
$K'(X)$ is defined as the $K$-theory of the
quasi-exact category of sheaves of partially cancellative
$\mathcal{O}_X$-sets (see \cite{HW} and \cite{CW} for details
of the definition and basic properties of $K'(X)$).

If $k$ is a commutative ring, we write $X_k$ for the scheme
which is the $k$-realization of $X$. In Section
\ref{sec:higher-K}, we prove that
the $K$-theory of a toric variety over any field can be recovered
from the $K$-theory of the underlying monoid scheme (see Theorem
\ref{thm:comparison}):

\begin{thm}\label{thm:comparison-intro}
Let $X$ be a toric monoid scheme, and $k$ a field.
Then the natural map of spectra is a weak homotopy equivalence:
\[
K(k)\wedge K' (X)\xrightarrow{\sim} K'(X_k)
\]
In particular, the realization map is an isomorphism:
\[ K'_0(X)\xrightarrow{\sim} K'_0(X_k). \]
\end{thm}

It follows from this theorem that our definition
\ref{def:K-of-matroid} of the $K$-theory of matroids
extends the existing definition of $K_0$, as described, for
example, in \cite{LLPP}.

In Section \ref{sec:Kzar}, we connect $K'(X)$ to the
$K$-theory spectrum $K^Q(X)$ of \cite{ELY}.
Although $K^Q$ 
is not appropriate for our purposes, its Zariski (homotopy)
sheafification $\Kzar(X)$ satisfies the formal properties we need
(see Theorems \ref{thm:K-vs-K'} and \ref{thm:scdh-descent}).

In section \ref{sec:proj-bundle-formula}, we verify a
projective bundle formula
for the $K'$-theory of smooth monoid schemes (see Theorem
\ref{thm:proj-bundle-formula}). Combining the results of sections
\ref{sec:Kzar} and \ref{sec:proj-bundle-formula}, we prove our second
main result (see Theorem \ref{thm:K-of-sm-proj}), where $\pi_n^s$
denotes the stable homotopy groups of spheres:
 
 \begin{thm}
Let $X$ be a smooth proper toric monoid scheme, and $k$ a field.
Then $K_0(X_k)$ is a free abelian group, independent of $k$,
and there are natural isomorphisms 
\[
\mu^X_n: K_0(X_k)\otimes_\Z \pi^s_n
= K_0(X_k)\otimes_\Z \Kzar_n(\pt) \xrightarrow{\sim} \Kzar_n(X).
\]
\end{thm}

In Section \ref{sec:samples}, we apply this theorem to compute
the $K$-theory of some sample matroids. Finally,
Appendix \ref{app:mult} collects some homotopy-theoretic
constructions needed in the body of the paper.

 \vspace{.1in}


\section{Monoid schemes.}\label{sec:toric}

By a {\it monoid} we mean a pointed commutative monoid,
i.e., a pointed set $A$ with basepoint $0$, equipped with a
associative, commutative product $A\times A \to A$ with
identity $1$ such that $0\cdot a = 0$ for all $a\in A$.
The notions of ideal, and prime ideal, make sense in a monoid $A$,
and we can form the prime ideal spectrum
$\MSpec(A)$, and hence monoid schemes,
just as in traditional ring theory and algebraic geometry.

 Here is some useful vocabulary. A monoid $A$ is {\it cancellative} if for $a, b, c \in A$
the conditions $ab = ac$ and $a\ne0$ together imply that $b = c$.
In this case, the unpointed monoid $A- \{0\}$ injects into
its group completion $A^+$,
and $\{0\}$ is the unique minimal prime ideal of $A$.
The {\it normalization} of a cancellative monoid $A$ is
$A_{nor}=\{a\in A^+: a^n\in A\}$
for some $n>0$, and $A$ is {\it normal} if $A=A_{nor}$.

A monoid is {\it partially cancellative and torsion free} (in short, {\it pctf}) if it is a quotient of a torsion free cancellative monoid by an ideal. A monoid scheme is pctf if its sheaf of monoids has pctf stalks; we write $\cMpctf$ for the category of separated, pctf monoid schemes of finite type. Given a commutative ring $k$, the $k$-monoid algebra functor extends to a $k$-realization functor $\cMpctf\to \Schk$, written $X\mapsto X_k$.

\begin{ex}
Let $A$ be the multiplicative monoid $\{0,1\}$, the initial object in the category of pointed
 commutative monoids. We write $\pt = \MSpec(A)$ for the associated monoid scheme, which is the final object in $\cMpctf$. We have $\pt_k = \Spec (k)$.
\end{ex}

\begin{defn}\label{def:TMS}  (\cite[4.1]{CHWW})
A {\it toric monoid scheme} is a separated, connected, torsionfree,
normal monoid scheme of finite type. 
\end{defn}

Recall that a {\it fan} consists of a free abelian group
$N$ of finite rank, together with a finite partially ordered set
$\Delta$ of {\it cones} (strongly convex rational cones $\sigma$
in $N_{\R}$); we refer the reader to \cite{Fulton} for details.

There is an essentially surjective, faithful functor $X$
from fans to toric monoid schemes, see \cite[4.2]{CHWW}. To describe it, we recall that the information of a monoid scheme $X$ can be encoded in a partially ordered set $(X,\leq)$ (the set of points of $X$, ordered by specialization) together with a {\it stalk functor} $A$ from $(X,\leq)$ to monoids; see \cite[Prop.\,2.11]{CHWW}.

\begin{defn}\label{fan-monoid-scheme}
Given a fan $(N,\Delta)$, we  define a contravariant functor $A$
from the poset $\Delta$ to monoids (written additively).
Set $M=\Hom(N,\Z)$.  
For each cone $\sigma $ in $\Delta$, form the monoid (pointed by adding a base point $\ast$):
\[
A(\sigma) = (\sigma^\vee\cap M)_\ast, \quad \textrm{where} \quad
\sigma^\vee = \{m\in M_\R: m(\sigma)\ge0\}.
\]
If $\tau$ is a face of $\sigma$, there is $m\in A(\sigma)$
  such that  $A(\tau)$ is obtained from $A(\sigma)$ by inverting $m$.
  This implies that there is a prime ideal $\p(\tau)$  of $A(\sigma)$
such that  $A(\tau)=A(\sigma)_{\p(\tau)}$.
As observed in \cite[4.2]{CHWW}, $A$ is the stalk functor of
a toric monoid scheme $X(N,\Delta)$;
when $N$ is understood, we just write $X(\Delta)$.
\end{defn}

\begin{rem}\label{rem:smooth-monoid-scheme}
	A monoid scheme $X$ is {\it smooth} if it is cancellative and its $k$-realization $X_k$ is a smooth $k$-scheme for all fields $k$. As discussed in \cite[Sec.\,6]{CHWW}, this is equivalent to the property that every stalk
	 $\mathcal{O}_{X,x}$ is a smash product of a pointed free monoid and a pointed free abelian group. In particular, if $X$ is affine, then $X\cong \A^s\times \G_m^t$ for some $s,t\geq 0$. 
\end{rem}

Given a monoid scheme $X$, an {\it equivariant closed subscheme} of $X$ is a closed subset $Y\subseteq X$ together with
 a sheaf of monoids $\mathcal{O}_Y$ that is a quotient of $\mathcal{O}_X$ by a quasi-coherent sheaf of monoid ideals.

\begin{ex}
	Let $X = X(\Delta, N)$ be a toric monoid scheme. The reduced and irreducible equivariant closed subschemes of $X$ are in one-to-one correspondence with the cones of $\Delta$. If $k$ is a field, and $Y\subseteq X$ is a reduced and irreducible equivariant closed subscheme, then $Y_k$ is an orbit closure in the toric $k$-variety $X_k$.
\end{ex} 

We conclude this section by recalling the notion of a {\it blow-up} of
monoid schemes. See \cite[Section 7]{CHWW} for a more detailed
discussion of projective morphisms of monoid schemes and blow-ups.

\begin{defn}\label{def:blowup}
Let $X$ be a monoid scheme of finite type and $Y \subseteq X$ an
equivariant closed subscheme, given by a quasi-coherent sheaf of monoid
ideals $\cI$. The {\it blow-up} of $X$ along $Y$ is the
projective morphism $X_Y=\MProj(\mathcal{O}_X[\cI t])\to X$.
If $k$ is a field, then $(X_Y)_k\to X_k$ is the
blow-up of the $k$-scheme $X_k$ along the closed subscheme $Y_k$.
\end{defn}

\section{$K'$ of monoid schemes and their $k$--realizations}
\label{sec:higher-K}

We now turn to the $K$-theory of pctf monoid schemes.
There are two definitions in the literature:
$K'(X)$, which is defined to be the $K$-theory of
the quasi-exact category of sheaves of finitely
generated $\pc$ $\mathcal{O}_X$-Sets; and $\KQ(X)$,
which is the $K$-theory
of the category of sheaves of projective sets; see \cite{HW,ELY}.  In
both cases, $K$-theory is defined as a connective spectrum using the
Q-construction of \cite{Quillen}, and we write $K'_n(X)$ for $\pi_n
K'(X)$, resp.\, $\KQ_n(X)$ for $\pi_n \KQ(X)$.

We begin by considering the $K'$-theory. The group $K'_0(X)$ has the following
elementary description  as a Grothendieck group, see \cite{HW}.

\begin{defn}
The group $K'_0(X)$ may be defined as a Grothendieck group: the abelian
group generated by the classes $[Z]$ of sheaves of finitely generated
$\pc$-sets, modulo the relations
that $[Z]=]Y]+[Z/Y]$ for every extension $Z\subset Y$;
	see \cite[2.1]{HW}
\end{defn}

\begin{ex}\label{ex:proj-space-K}
  $K'_0$ of the projective space $\bbP^n$ was computed in \cite[5.7]{HW}:
  $K'_0(\bbP^n)\cong \Z^{n+1}$ on the classes of the structure sheaves
  of the closed	subvarieties $\bbP^i$ ($i=0,...,n$). 
  It follows that the canonical map (induced by realization)
  $K'_0(\bbP^n) \to K'_0(\bbP^n_k)$ is an isomorphism for any field $k$.
\end{ex}

Recall from \cite[5.1.2]{HW} that $K'$ is contravariantly functorial
for open immersions, and covariantly functorial for equivariant closed
immersions. We will need the following result, taken from
\cite[Thm. 5.3]{HW}. In {\it loc.\,cit.\,}, this is proved using
Devissage \cite[Thm. 3.2]{HW} and the localization theorem of Campbell
and Zakharevich for ACGW categories \cite[8.5]{CZ}. A more streamlined
proof along the lines of Quillen's localization theorem in
\cite[Section 7]{Quillen} replaces the use of ACGW categories with
arguments purely within the world of (regular) quasi-exact categories,
see \cite[Section 4]{CW}.

\begin{thm}\label{thm:K-loc}
	Let $X$ be a pc monoid scheme and $Z\map{i} X$ an equivariant closed subscheme
	with open complement $U\map{j} X$. Then there is a fibration sequence of connective spectra
	\[
	K'(Z)\map{i_*} K'(X) \map{j^*} K'(U).
	\]
In particular, the map $K'_0(X)\to K'_0(U)$ is onto. 
\end{thm}

\begin{cor}\label{cor:K'-MV}
  Let $X$ be a pc monoid scheme and $U,\,V\subseteq X$
  open subschemes covering $X$. 
	Then there is a long exact Mayer--Vietoris sequence
\[
\dotsm \to K'_n(X)\to K'_n(U)\oplus K'_n(V)\to K'_n(U\cap V) \to K'_{n-1}(X)\to \dotsm
\]
\end{cor}

\begin{proof}
  This is a standard argument, see for example
  \cite[Prop.\,4]{BG} or \cite[V.6.11.2]{WK}.
\end{proof}

\begin{rem}
We also have $K'_0(X\times\Gm)\simeq K'_0(X)$; see \cite[5.4]{HW}.
\end{rem}

We are now in position to prove that the $K'$-theory of a toric
variety can be computed from that of the underlying monoid
scheme. Recall that, if $Y$ is a $k$-scheme, then the $K'$-theory
spectrum $K'(Y)$ is a module over the $E_\infty$-ring spectrum
$K(k)$. It follows from this that, for a monoid scheme $X$, there is a
natural map of spectra $K(k)\wedge K' (X)\to K'(X_k)$ given as the
composite of the $k$-realization with the action map.

\begin{thm}\label{thm:comparison}
Let $X$ be a toric monoid scheme, and $k$ a field.
Then the natural map of spectra 
\[ K(k)\wedge K' (X)\xrightarrow{\sim} K'(X_k) \]
is a weak homotopy equivalence. 
In particular, the realization map is an isomorphism:
\[ K'_0(X)\xrightarrow{\sim} K'_0(X_k). \]
\end{thm}

\begin{proof}
  We proceed by induction on the dimension $n$ of $X$.
  If $n = 0$, then $X$ is a point, and $K'(X)$ is the
  sphere spectrum $\mathbb{S}$. On the other hand,
  $X_k = \Spec (k)$, so $K'(X_k) = K'(k)$.
  The asserted equivalence follows. 
	
  Now assume that $n > 0$, and we have proved the assertion for all
  toric monoid schemes of dimension less than $n$. Since $X$ is a
  toric monoid scheme, there is a sequence of open subschemes 
\[
\MSpec(\Z_+^{\wedge t}) \cong X_0 \subset X_1 \subset \dotsm \subset X_t = X
\]
such that, for $0 < i\leq t$, $Z_i = X_i \setminus X_{i-1}$ is a toric
monoid scheme of dimension $i-1$ embedded as an equivariant closed
subscheme into $X_i$. The assertion of the theorem is true for $X_0$
by \cite[Thm.\,4.6]{HW} and the Fundamental Theorem of $K'$-theory
(see \cite[Section 6, Thm.\,8]{Quillen}). It also holds for the monoid
schemes $Z_i$ by the inductive hypothesis. Now it follows for $X_i$,
$i > 0$, using induction on $i$ and Theorem \ref{thm:K-loc}.
	
	The realization map $K'_0(X)\cong K'_0(X)\otimes_\Z K_0(k)\to K'_0(X_k)$ is an edge map in the K\"unneth spectral
	sequence (see \cite{Adams69}) converging to $K'_*(X_k)\cong \pi_*\left(K(k)\wedge K' (X)\right)$. Since all spectra involved are connective, the map is an isomorphism.
\end{proof}

\begin{rem}\label{rem:higher-comparison}
The statement of Theorem \ref{thm:comparison} applies
to general pc monoid schemes $X$, using localization and devissage.
We omit the details here, as we do not need that level of generality. 
\end{rem}

\section{The $K$-theory of monoid schemes.}\label{sec:Kzar}

In Theorem \ref{thm:K-of-sm-proj}, we will determine the $K'$-theory
of a smooth proper toric monoid scheme. Along the way, we establish
projective bundle and blow-up formulas. To formulate and prove these
results, we need to identify $K'(X)$ with an appropriate $K$-theory
spectrum of $X$ that extends to a presheaf on toric monoid
schemes. Unfortunately, the definition of $\KQ(X)$ in \cite{ELY} (as
the $K$-theory of the category of sheaves of projective
$\mathcal{O}_X$-sets) does not satisfy Zariski descent, and fails to
be isomorphic to $K'(X)$ for general smooth toric monoid schemes. (For
an example of this failure, compare the calculation in \cite[Ex
  3.20]{ELY} with our Example \ref{ex:proj-space-K}.) To fix this, we
introduce a new functor $\Kzar$, defined by sheafifying the presheaf
$\KQ$ in the Zariski topology.

Recall from \cite{CHWW-p} that  the category of toric monoid schemes is a subcategory of the category $\cMpctf$ of
partially cancellative, torsionfree monoid schemes of finite type. 

As described for example in \cite[Sec.\,12]{CHWW} and \cite[Sec.\,5]{CHWW-p}, there is a model category structure on the
category of presheaves of spectra on $\cMpctf$ whose weak equivalences are those maps inducing isomorphisms between the (Zariski) sheaves of stable homotopy groups. Given a presheaf of spectra $\mathcal{E}$, we write $\mathcal{E}\to\bbH_{\zar}(-, \mathcal{E})$ for its fibrant replacement in this model category structure. We say that $\mathcal{E}$ {\it satisfies Zariski descent} if, for every $X\in \cMpctf$, the map of spectra 
\[\mathcal{E}(X)\xrightarrow{\sim} \bbH_{\zar}(X, \mathcal{E})\] 
is a weak homotopy equivalence. 

An analogous model category structure exists for the category of presheaves $\mathcal{E}$ of spectra on the small Zariski site of a fixed monoid scheme $X$. We again write $\mathcal{E}\to \bbH_{\zar}(-, \mathcal{E})$ for the fibrant replacement in this model category structure. Since the restriction of a fibrant presheaf from $\cMpctf$ to $X_\zar$ is fibrant, this does not create any confusion. We say that $\mathcal{E}$ satisfies descent on $X_\zar$ if, for all $U\subseteq X$ open, the map $\mathcal{E}(U)\to \bbH_{\zar}(U, \mathcal{E})$ is a weak homotopy equivalence.

\begin{lem}\label{lem:K'-descent}
	Let $X\in\cMpctf$. Then the presheaf $K'$ on $X_\zar$ satisfies descent. 
\end{lem}

\begin{proof}
	By \cite[Lemma 5.2]{CHWW-p}, this is true when $X$ is affine. Now the general case follows using Corollary \ref{cor:K'-MV}.
\end{proof}

\begin{defn}\label{defn:K-Zar}
Let $\KQ$ be the presheaf of $K$-theory spectra on $\cMpctf$ 
defined in \cite[Section 3.4]{ELY}.
We write $\Kzar$ for the presheaf $\bbH_{\zar}(-, \KQ)$. For $n\in\Z$ we set
$\Kzar_n(X) = \bbH_{\zar}^{-n}(X, \KQ).$
\end{defn}

By construction, there is a natural transformation $\KQ\to \Kzar$.
Here is another application of \cite[Lemma 5.2]{CHWW-p}:

\begin{lem}\label{lem:affine-equivalence}
Let $U\in \cMpctf$ be affine. Then $\KQ(U)\rightarrow \Kzar (U)$
is a weak homotopy equivalence.
Therefore $\KQ_n(U)\xrightarrow{\sim} \Kzar_n(U)$
is an isomorphism for all $n$.
\end{lem}

\begin{lem}\label{lem:K-htpy-invariance}
	If $A$ is a pc monoid, then the natural maps
	 \[\KQ(A)\to \KQ(A\wedge \N_+)\quad\text{and}\quad K'(A)\to K'(A\wedge \N_+)\]
	 are weak homotopy equivalences. 
\end{lem}

\begin{proof}
	The first assertion follows from \cite[Thm.\,2.6]{ELY} since the natural homomorphism of units $A^\times \to \left(A\wedge \N_+\right)^\times$ is an isomorphism. The second assertion is \cite[Thm.\,4.6]{HW}.
\end{proof}
 For a pctf monoid scheme $X$, consider the map
$\KQ(X)\to K'(X)$ induced by the inclusion of the category of
sheaves of projective $\mathcal{O}_X$-sets into that of
sheaves of pc $\mathcal{O}_X$-sets. We claim that this map factors through a natural map $\Kzar(X)\to K'(X)$
up to a canonical weak homotopy equivalence. Indeed, $K'$
satisfies Zariski descent on $X$ by Lemma \ref{lem:K'-descent}. Thus, $\KQ(X)\to K'(X)$
factors as
\[  \KQ(X)\to \Kzar(X) = \bbH_{\zar}(X, \KQ) \to \bbH_{\zar}(X, K')\simeq K'(X).   \]

\begin{thm}\label{thm:K-vs-K'}
	Let $X$ be a smooth monoid scheme. Then the natural map
 $\Kzar (X)\xrightarrow{\sim} K'(X)$ is a weak homotopy equivalence.
\end{thm}

\begin{proof}	
	 Assume first that $X = \MSpec (A)$ is affine and smooth. In this case, $\KQ(X)\xrightarrow{\sim} \Kzar(X)$ by Lemma \ref{lem:affine-equivalence}, so we need to show that the map $\KQ(X)\to K'(X)$ is a weak homotopy equivalence. Since $X$ is smooth, it is a product of an affine space and a torus, in other words, $A$ is isomorphic to a smash product $\N_+^{\wedge s} \wedge \Z_+^{\wedge t}$  of a
	free pointed abelian monoid and a free pointed abelian group. Suppose that $s = 0$. Then $A$ is a (pointed) group, and thus, every pc $A$-set is automatically projective (as noted in \cite[Ex.\,1.4]{HW}). It follows that $\KQ(X) = K'(X)$.
	
	 Now suppose $s > 0$. Consider the following commutative square of $K$-theory spectra:
		\begin{equation}\label{cart-square}
		\xymatrix{
			\KQ(\Z_+^{\wedge t})  \ar[r]\ar[d]^{\cong}&
			K'(\Z_+^{\wedge t}) \ar[d]^{\cong} 
			\\
			\KQ(\N_+^{\wedge s} \wedge \Z_+^{\wedge t})  \ar[r]& K'(\N_+^{\wedge s} \wedge \Z_+^{\wedge t}).
		}
	\end{equation}
	We have just observed that the top horizontal map is a weak homotopy equivalence. Since the
	 vertical maps are weak homotopy equivalences by Lemma \ref{lem:K-htpy-invariance}, the bottom horizontal map is also a weak homotopy equivalence. 
	
	The general case now follows by Zariski descent, which holds for $K'$ by Lemma \ref{lem:K'-descent}, and for $\Kzar$ by definition.
\end{proof}	

\begin{cor}\label{cor:K_0}
	Let $X$ be a smooth monoid scheme, and $k$ a field. Then there is a natural isomorphism
	$\Kzar_0(X)\xrightarrow{\sim} K_0(X_k)$.
\end{cor}

\begin{proof}
	Combine Theorems \ref{thm:K-vs-K'} and \ref{thm:comparison}.
\end{proof}

Let $X = \MSpec (A)$ be a smooth, affine, toric monoid scheme, 
and $i:Y\hookrightarrow X$ a smooth equivariant closed subscheme. Observe that there exist natural numbers $r$, $s$, $t$ such that the immersion $i:Y\hookrightarrow X$ is isomorphic to the immersion $\{0\}\times \A^s\times \G_m^t \hookrightarrow \A^r\times \A^s\times \G_m^t$. It follows that $X$ is a vector bundle over $Y$.

Now let $\pi: X'\to X$ be the blow-up of $X$ along $Y$, and $i':Y' = \pi^{-1}(Y)\hookrightarrow X'$ the exceptional divisor. Then $X'$ is a line bundle over $Y'$ with zero section $i'$,
 because the blow-up of the affine space $\A^r$ at the origin is the total space of the tautological line bundle over $\bbP^{r-1}$ and $X'$ is the product of this blow-up with $\A^s\times \G_m^t.$

\begin{lem}\label{lem:local-embedding}
	The pullback maps $i^*:\Kzar(X)\to \Kzar(Y)$ and $i'^*:\Kzar(X')\to \Kzar(Y')$ are weak
	 homotopy equivalences. 
\end{lem}

\begin{proof}
	$\Kzar$ is homotopy invariant for vector bundles by Lemma \ref{lem:K-htpy-invariance} and Zariski descent. Now the
	 assertion follows from the observations in the two paragraphs above.
\end{proof}

\goodbreak
\begin{thm}\label{thm:scdh-descent}
	Let $X$ be a smooth toric monoid scheme, $i:Y\hookrightarrow X$ a smooth closed equivariant subscheme, $X'\xrightarrow{\pi} X$ the blow-up along $Y$ and $Y'\xrightarrow{\pi'} Y$ the exceptional divisor. Let $U\subseteq X$ be open,
	 and write $U' = \pi^{-1}(U)$. Then the following commutative square of spectra is homotopy cartesian:
	\begin{equation}\label{MV-square}
		\xymatrix{
			\Kzar(U)  \ar[r]^{i^*}\ar[d]^{\pi^*}&
			\Kzar(U\cap Y) \ar[d]^{\pi'^*} 
			\\
			\Kzar(U')  \ar[r]^{i'^*}& \Kzar(U'\cap Y').
		}
	\end{equation}
\end{thm}

\begin{proof}
Write $F(U) = \mathrm{hofib}(i^*\vert_U)$ and $F'(U) = \mathrm{hofib}(i'^*\vert_{\pi^{-1}(U)})$ for the (horizontal) homotopy fibers evaluated at $U$. Then $F$ and $F'$ are presheaves of spectra on $X$ satisfying Zariski descent, and the commutativity of the square induces a natural transformation $\phi:F\to F'$. Assume $V$ is affine. If $V\cap Y = \emptyset$, then $V'\to V$ is an isomorphism and $F(V)\to F'(V)$ is a weak homotopy equivalence. On the other hand, if $V\cap Y\neq \emptyset$, then Lemma \ref{lem:local-embedding} implies that $F(V)\simeq \ast\simeq F'(V)$. Therefore, $\phi(V)$ is a weak homotopy equivalence for all affine open $V\subseteq X$. 

By Zariski descent, $F(U)\to F'(U)$ is a weak homotopy equivalence for all open $U\subseteq X$. That is, \eqref{MV-square} is homotopy cartesian. 
\end{proof}

\section{Projective bundle formula.}\label{sec:proj-bundle-formula}

Let $E\to X$ be a (geometric) vector bundle of rank $r+1$ over a smooth monoid scheme $X$, and $p:\bbP = \bbP(E)\to X$ be the associated projective bundle. 
Write $z$ for the class of the tautological bundle $\mathcal{O}_{\bbP_k}(-1)$ in $K_0(\bbP_k)$. 
Let $\Z^S$ denote the free abelian group on the set $S = \{z^0,z^1,\cdots, z^r\}$, and $k$ a field. By \cite[VI.1.1]{SGA6}
\[K_0(\bbP_k) \cong \bigoplus_{m=0}^r K_0(X_k)\cdot z^m  = \Z^{S}\otimes_\Z K_0(X_k)\] 
as a $K_0(X_k)$-module; these isomorphisms are natural with respect to pullback. Combining with the flat pullback $p^*: K'_n(X)\to K'_n(\bbP)$, the action of Corollary \ref{cor:vb-action} gives a natural homomorphism 
\begin{equation}\label{phi-defn}
\phi^r:\Z^{S}\otimes_\Z K'_n(X)\to K'_n(\bbP)
\end{equation}
sending $v\otimes \alpha$ to $v\cup p^*(\alpha)$.

\begin{thm}\label{thm:proj-bundle-formula}
	Let $n\geq 0$ and $r\geq 0$. The natural homomorphism $\phi^r$ of \eqref{phi-defn} is an
	 isomorphism:
	\[\Z^{S}\otimes_\Z K'_n(X)\xrightarrow{\phi^r} K'_n(\bbP).\]
	 Equivalently, 
	 there are natural isomorphisms, 
	\[\phi^E:\Kzar_0(\bbP)\otimes_{\Kzar_0(X)} K'_n(X)\to K'_n(\bbP).\]
\end{thm}

Our proof follows Quillen's argument in \cite[Section 7, 4.3]{Quillen}, reducing to the case $X = \pt$, see also \cite[Prop.\,5.7]{HW}.

\begin{prop}\label{prop:proj-space}
	Let $n,\,r\geq 0.$ Then the map $\phi^r$ of \eqref{phi-defn} is an isomorphism
		\[\Z^{S}\otimes_\Z K'_n(\pt)\xrightarrow{\sim} K'_n(\bbP^r).\]		
\end{prop}

\begin{proof}
	For any field $k$, $K_0(\bbP^r_k)$ is isomorphic to $A_r = \Z[t]/(t^{r+1})$ as a ring, where $t = 1-z$. 
	We regard $\phi^r$ as a map $A_r\otimes_\Z K'_n(\pt)\to K'_n(\bbP^{r})$. 
	To prove that $\phi^r$ is an isomorphism for all $r$, we proceed by induction on $r\geq 0$. 
	
  The assertion is trivial if $r = 0$. Suppose now that $r > 0$ and we have proved that $\phi^{r-1}: A_{r-1}\otimes_\Z K'_n(\pt)\to K'_n(\bbP^{r-1})$ is an isomorphism. Recall from \cite{HW} that $K'_n(\pt) \cong \pi^s_n$, the $n$-th stable stem. 
  
  Let $\bbP^{r-1}\xrightarrow{i} \bbP^r$ be the usual closed immersion, and $\A^r\xrightarrow{j} \bbP^r$ the complementary open immersion. Consider the following (not obviously commutative) diagram:
	\begin{equation}\label{projective-ladder}
		\xymatrix{
			0  \ar[r]&
			A_{r-1}\otimes_\Z \pi^s_n \ar[d]^{\phi^{r-1}} \ar[r]^{t\cdot}& A_r\otimes_\Z \pi^s_n \ar[d]^{\phi^r} \ar[r]^{ev_0}& \pi^s_n \ar[d]^{\lambda^*} \ar[r]& 0
			\\
			0  \ar[r]&
			K'_n(\bbP^{r-1}) \ar[r]^{i_*}& K'_n(\bbP^r) \ar[r]^{j^*}& K'_n(\A^r) \ar[r]& 0
		}
	\end{equation}
	
	The first vertical map $\phi^{r-1}$ is an isomorphism by induction, and the third vertical map $\pi^s_n \cong K'_n(\pt)\to K'_n(\A^r)$ is given by flat pullback, which is an isomorphism by the homotopy invariance of $K'_\ast$. The top row in this diagram is (split) exact. 
	
	The composition of flat pullbacks $K'_n(\pt)\to K'_n(\bbP^r)\to K'_n(\A^r)$ is an isomorphism by homotopy invariance, and thus the restriction map $j^*:K'_n(\bbP^r)\to K'_n(\A^r)$ is split surjective for all $n\geq 0$; it follows that the bottom row of the diagram is (split) exact for all $n$ as well. To conclude the proof, it therefore suffices to show that the diagram \eqref{projective-ladder} commutes. 
	
	Consider the left-hand square first. Unravelling definitions, we must show the following:
	
	\noindent{\it Claim:} Write $q_r:\bbP^r\to \pt$ and $q_{r-1}: \bbP^{r-1}\to \pt$ for the structure maps, and $i:\bbP^{r-1}\to\bbP^r$ for the closed immersion. Let $k$ be a field, $y\in K_0(\bbP^{r-1}_k)\cong K'_0(\bbP^{r-1})$, and $\alpha\in K'_n(\pt)$. Then
	\[ i_*(y\cup q_{r-1}^*(\alpha)) = i_*(y)\cup q_r^*(\alpha). \]
	Indeed, this is a special case of Lemma \ref{lem:projection_formula_1} below. 
	
	Now consider the right-hand square of \eqref{projective-ladder}, writing $ev_0$ for the evaluation at $0$ map and $\lambda:\A^r\to \pt$ for the structure map, so that the right vertical isomorphism acts on $\alpha\in\pi^s_n\cong K'_n(\pt)$ as $\alpha\mapsto \lambda^*(\alpha)$. We have to show that $j^*\circ\phi^r = \lambda^*\circ ev_0$. By commutativity of the left square, both composites are zero on elements of the form $t^m\otimes\alpha$ with $m > 0$. If $m = 0$, then by functoriality of flat pullback:
	$j^*(\phi^r(1\otimes\alpha))  = j^* (q_r^*(\alpha)) = \lambda^*(\alpha)$
\end{proof}

\begin{proof}[Proof of Theorem \ref{thm:proj-bundle-formula}]
	As observed in \cite[4.3]{Quillen}, the assertions that $\phi^r$ and $\phi^E$ are isomorphisms are equivalent. Consider the presheaves of sequences of abelian groups on $X$ sending $j:U\hookrightarrow X$ to the sequences 
	\[\mathbf{A}_n(U) = \Z^{S}\otimes_\Z K'_n(U)\; \text{and}\; 
	\mathbf{B}_n(U) = K'_n(\bbP(j^* E))\]
	 in the theorem, respectively. Then $\phi^r$ induces a natural transformation between these presheaves. 
	
	Since $\Z^{S}$ is a free abelian group, Corollary \ref{cor:K'-MV} implies that $\mathbf{A}_*$ has a long exact Mayer-Vietoris sequence for open covers. By the same result, $\mathbf{B}_*$ satisfies Mayer-Vietoris since, for an open immersion $j:U\hookrightarrow X$, $\bbP(j^* E)\to \bbP(E)$ is also an open immersion. Using induction on the cardinality of an affine open cover of $X$, we reduce to the case where $X$ is affine. Now since $X$ is smooth, 
	$X = \MSpec(\N_+^{\wedge s}\wedge \Z_+^{\wedge t})$. This implies in particular that $p:E\to X$ is a trivial vector bundle. Applying homotopy invariance and the Fundamental Theorem \cite[Theorem 4.6]{HW}, we can further reduce to the case where $X$ is a point, so that $\bbP(E) \cong \bbP^{r}$. Now Proposition \ref{prop:proj-space} applies to finish the proof. \end{proof}

\section{$K$-theory of smooth proper monoid schemes.}\label{sec:K-of-sm-proj}

In this section, we show that the $\Kzar$-theory of a smooth proper
toric monoid scheme $X$ is entirely determined by $\Kzar_0(X)$. Assume
that $\Kzar_0(X)$ is a free abelian group. Then there is a natural map
of ring spectra (see \ref{cor:algebras})
\begin{equation}\label{eqn:K_0-action}
	\mu^X: \Kzar_0(X)\otimes \mathbb{S}\to \Kzar(X)
\end{equation}
which, on homotopy groups, induces the homomorphisms
\[
\mu^X_n: \Kzar_0(X)\otimes \pi^s_n = \Kzar_0(X)\otimes \Kzar_n(\pt) \to \Kzar_n(X)
\]
given by $\mu^X_n(v\otimes\alpha) = v\cup \lambda^*(\alpha),$
where $\lambda$ is the structure map $X\to\pt$.
	
\begin{thm}\label{thm:K-of-sm-proj}
  Let $X$ be a smooth proper toric monoid scheme, and $k$ a field.
  Then $\Kzar_0(X)$ is a free abelian group, and the natural map
  of \eqref{eqn:K_0-action} is a weak homotopy equivalence.
  Equivalently, the natural homomorphisms 
  \[ \mu^X_n: \Kzar_0(X)\otimes_\Z \pi^s_n
  = \Kzar_0(X)\otimes_\Z \Kzar_n(\pt) \xrightarrow{\sim} \Kzar_n(X)
  \]
are isomorphisms for all $n\geq 0$.
\end{thm}

Our proof relies on the factorization theorem of Morelli
\cite{Morelli} and W\l{}odarczyk \cite{Wlodarczyk}. The following
statement suffices for our purposes, and is an immediate consequence
of the main result of \cite{Wlodarczyk}:
	\begin{prop}\label{prop:factorization}
		Let $X_1$ and $X_2$ be smooth proper toric monoid schemes of the same dimension. Then there exists a smooth proper toric monoid scheme $\widetilde{X}$ and morphisms $f_1:\widetilde{X}\to X_1$ and $f_2:\widetilde{X}\to X_2$ such that $f_1$ and $f_2$ are compositions of blow-ups along smooth closed equivariant subschemes. 
	\end{prop}
	
\begin{proof}
The fans of $X_1$ and $X_2$ are both regular and complete by
assumption. Thus, they satisfy the hypotheses of
\cite[Thm.\,A]{Wlodarczyk}. Now the assertion follows by applying the
functor from fans to toric monoid schemes of Definition
\ref{fan-monoid-scheme}.
\end{proof}
Now let $X$ be a smooth proper toric monoid scheme of dimension $d$,
        and let $i:Y\hookrightarrow X$ be a
 smooth equivariant closed subscheme. Consider the blow-up square
\begin{equation}\label{blow-up-square}
	\xymatrix{
		Y'  \ar[r]^{i'}\ar[d]^{\pi'}&
		X' \ar[d]^{\pi} 
		\\
		Y  \ar[r]^{i}& X.
	}
\end{equation}
Note that $Y'\xrightarrow{\pi'} Y$ is a projective bundle. 
	
\begin{prop}\label{prop:single-blow-up}
  Assume that $\Kzar_0(Y)$ is free abelian and the map $\mu^Y_n$
  of \eqref{eqn:K_0-action} is an isomorphism for all $n\geq 0$. \\
  Then $\Kzar_0(X)$ is free and $\mu^X_n$ is an isomorphism for all
  $n\geq 0$ if, and only if, $\Kzar_0(X')$ is free and
  $\mu^{X'}_n$ is an isomorphism for all $n\geq 0$. 
\end{prop}

\begin{proof}
It follows from the hypothesis on $Y$ and the projective bundle
formula (Theorem \ref{thm:proj-bundle-formula}) that
$\Kzar_0(Y')$ is free, and
$\mu^{Y'}_n$ is an isomorphism for all $n\geq 0$.
By \cite[Thm.\,2.1]{T} and Corollary \ref{cor:K_0},
\[\Kzar_0(X')\cong \Kzar_0(X)\oplus \Kzar_0(Y)^{\oplus (c-1)},\]
where $c$ is the codimension of $Y$ in $X$.
Since $\Kzar_0(Y)$ is assumed free, it follows that
$\Kzar_0(X)$ is free if, and only if, $\Kzar_0(X')$ is. 
	
Let $k$ be a field. Then \cite[Lem.\,2.3]{T} implies that 
\[\Kzar_0(X)\cong K_0(X_k)\to K_0(X'_k)\cong \Kzar_0(X')\]
is a split injection. Since $Y$ is smooth, $K_{-1}(Y_k) = 0$.
Therefore the Mayer-Vietoris sequence
\begin{equation}\label{MV-split} 
0\to \Kzar_0(X)\to \Kzar_0(Y)\oplus \Kzar_0(X')\to \Kzar_0(Y')\to 0 
\end{equation}
is split exact. 
	
Now consider the following diagram, where $n\geq 0$:
\begin{equation*}\xymatrix{
\Kzar_0(X)\otimes_\Z \pi^s_n \ar[d]^{\mu^X_n} \ar[r]^{\mathrm{into}}&
\left(\underset{\Kzar_0(Y)}{\overset{\Kzar_0(X')}{\oplus}}\right)
\otimes_\Z \pi^s_n \ar[d]^{\mu^Y_n\oplus\mu^{X'}_n} \ar[r]& \Kzar_0(Y')\otimes_\Z\pi^s_n \ar[d]^{\mu^{Y'}_n} \ar[r]& 0
			\\
\Kzar_n(X) \ar[r]& \left(\underset{\Kzar_n(Y)}{\overset{\Kzar_n(X')}{\oplus}}\right) \ar[r]&\Kzar_n(Y') \ar[r]^{\delta_n}& \Kzar_{n-1}(X) 
}
\end{equation*}

The bottom row is a piece of the long exact Mayer-Vietoris sequence
associated to the homotopy cartesian square \eqref{MV-square},
and the top row is the split exact sequence \eqref{MV-split},
tensored with $\Kzar_n(\pt) = \pi^s_n$.
The diagram commutes, since the action of Corollary \ref{cor:vb-action}
is natural. We observed above that, for all $n\geq 0$,
the right vertical homomorphism $\mu^{Y'}_n$ is an isomorphism; it follows by an easy diagram chase that $\delta_n = 0$, and thus the sequence
\[
  0\to \Kzar_n(X)\to \Kzar_n(Y)\oplus \Kzar_n(X')\to \Kzar_n(Y')\to 0
\]
is short exact for all $n\geq 0$ as well.
Now the Five Lemma implies that $\mu_n^X$ is an isomorphism
exactly when
$\mu_n^{X'}$ is. 
\end{proof}
	
\begin{proof}[Proof of Theorem \ref{thm:K-of-sm-proj}]
We proceed by induction on $d = \dim (X)$. If $d = 0$,
then $X = \pt$ and the assertion is straightforward.
Now let $d > 0$, and assume we have proved that $\Kzar_0(Y)$ is a
free abelian group and the map $\mu^Y$ of \eqref{eqn:K_0-action}
  is a weak homotopy equivalence for all smooth proper toric
  monoid schemes $Y$ of dimension less than $d$. 
		
  By Proposition \ref{prop:factorization}, there is a smooth proper
  toric monoid scheme $\tilde{X}$ and morphisms $\tilde{X}\to X$
  and $\tilde{X}\to\bbP^d$ that are compositions of blow-ups
  along smooth closed equivariant subschemes. Now $\mu^{\bbP}$,
  where $\bbP = \bbP^d$, is a weak homotopy equivalence by
  Proposition \ref{prop:proj-space}, and $\Kzar_0(\bbP^d)$ is
  free of rank $d+1$. 
  By induction on the number of blow-ups needed in the factorization,
  we conclude (using  Proposition \ref{prop:single-blow-up}) that
  $\Kzar_0(X)$ is a free abelian group, and that $\mu^X_n$ is an
  isomorphism for all $n\geq 0$. 
\end{proof}
	
\begin{rem}
  Using Theorem \ref{thm:comparison}, this provides
  an alternative proof for the proper case of
  Vezzosi and Vistoli \cite[Cor.\,6.10]{VV}.
\end{rem}

\section{The higher $K$-theory of matroids.}\label{sec:samples}

In this section, we compute the $K$-theory of some sample matroids,
as defined in Definition \ref{def:K-of-matroid}.
First, we briefly recall the relevant terminology.

\begin{defn}\label{lin-flat-defn}
A matroid is a set $E$ together with a collection 
$\cF$ of subsets called ``flats" satisfying the following properties:
\begin{enumerate}
\item[(i)] $E\in\cF$.  
\item[(ii)] $F,G\in\cF$ implies $F\cap G\in\cF$.
\item[(iii)] If $F$ is a flat, and $x\in E\setminus F$, 
  then there is a unique flat $F'$ minimal over $F$ containing $x$. 
	\end{enumerate}	
\end{defn}

\begin{ex}\label{ex:uniform} 
Let $E=\{1,2,...,n\}$ and $m\le n$. In the 
{\it uniform matroid} $U_{m,n}$ (see \cite[1.2.7]{Ox})
the flats are all subsets of $E$ having cardinality
less than $m$, and the whole set.
\end{ex}

If $(E,\cF)$ is a matroid, a {\em loop} is an element $x\in E$ such that $\{x\}$ is contained in every flat. A matroid is called {\em loopless} if it has no loops. To a matroid, one can associate a fan and hence, a monoid scheme (see Section \ref{sec:toric}):

\begin{defn}\label{Bergman}  (\cite[3.10]{Huh})
	Given a loopless matroid $(E,\cF)$ of rank $r+1$, 
	the {\it Bergman fan} $\Delta_{\cF}$
	is the $r$-dimensional fan in $N=\Z^{|E|}/(1,...,1)$ whose cones
	$\sigma_{\cF}$ correspond to
	flags of nonempty proper flats of $(E,\cF)$.
	\[
	\cF = (F_1 \subset F_2 \subset \cdots \subset F_d).
	\]
	More precisely, fix a bijection $E \cong \{0,1,\dotsc, n\}$ and let 
	\[
	\bu_0=(1,0,\dots,0),...,\bu_n=(0,\dots,0,1)\quad \text{in~}\, \Z^{n+1}.
	\]
	be the $n+1$ standard unit vectors in $\Z^{n+1}.$
	
	For a subset $S$ of $E = \{0,\dotsc, n\}$, let $\bu_S\in N$ be the
	image of the vector $\sum_{i\in S} \bu_i$. Then the cone associated to
	the flag $\cF$ is spanned by $\{\bu_{F_1},\dotsc, \bu_{F_d}\}.$ In
	particular, the set of rays of the Bergman fan is in one-to-one
	correspondence with the set of flats of the matroid.
	
	If $E$ is not loopless,  write $E=E'\amalg E''$,
	where $E'$ is loopless and $E''$ is the set of loops in $E$.
	Then the Bergman fan of $E$ is defined as the product of the
	Bergman fan of $E'$, and the fan corresponding to the
	one-point monoid scheme $\MSpec(\Z^r)$, with $r=|E'']$.
\end{defn}

\begin{defn}\label{defn:permutahedron}
	(See \cite[3.2]{Huh}.)
	The $n$-dimensional permutohedral fan in
	$N=\Z^{n+1}/\text{span}(1,1,\dots,1)$ is the complete fan
	whose cones are spanned by $\{\bu_{S_1}, \cdots, \bu_{S_d}\}$,
	where 
	\[\emptyset\subset S_1\subset S_2 \subset \dotsm \subset S_d\subset \{0,1,\dotsc, n\}\]
	 is a flag of proper, non-empty subsets. 
This is the Bergman fan of the Boolean matroid $U_{n+1,n+1}$.
\end{defn}

The Bergman fan of a general loopless matroid with
underlying set $E = \{0,1,\dotsc, n\}$ is, by construction,
a subfan of the $n$-dimensional permutohedral fan. 
The following well--known proposition is an immediate consequence
of this fact, and Definition \ref{defn:permutahedron}. 

\begin{prop}\label{bergman-fans-smooth}
The toric monoid scheme 
associated to the Bergman fan of a matroid $(E,\cF)$ is smooth and
quasiprojective.
\end{prop}

Recall from Definition \ref{def:K-of-matroid} that the $K$-theory
spectrum of a matroid $(E,\cF)$ is defined as $\Kzar(X)$, where $X$ is
the toric monoid scheme associated to the Bergman fan of
$(E,\cF)$. Since the permutohedral fan is complete, Theorem
\ref{thm:K-of-sm-proj} and the explicit presentation of $\Kzar_0(X) =
K_0(X_\mathbb{C})$ given in \cite[Thm.\,5.2 and page 11]{LLPP} allow
us to determine the $K$-theory of the Boolean matroid:
\begin{ex}\label{ex:K-of-boolean}
	Fix an integer $n\geq 2$. For all $m\geq 0$, we have
	\[K_0(U_{n,n})\otimes_\Z\pi_m^s\xrightarrow{\sim} K_m(U_{n,n}).\]
	 Moreover, $K_0(U_{n,n}) = \Z[\{x_F\}, F\in \mathcal{P}_n]$, where $\mathcal{P}_n$ is the set of non-empty proper subsets of $\{1,2,\dotsc, n\}.$ Thus, the elements of $\mathcal{P}_n$ are in one-to-one correspondence with the rays of the
		 $(n-1)$-dimensional permutohedral fan. The generators $x_F$ satisfy the following list of relations:
	\begin{enumerate}
		\item $1 - \displaystyle\prod_{F} (1 - x_F)$;
		\item $1 - \displaystyle\prod_{i\notin F} (1 - x_F)$, for $i = 1,\dotsc, n$;
		\item $x_F x_G$, for all incomparable subsets $F$ and $G$ of $\{1,\dotsc, n\}.$ 
	\end{enumerate}
\end{ex}

The Bergman fan of a matroid $(E,\cF)$ that is not Boolean will not be
a complete fan, so Theorem \ref{thm:K-of-sm-proj} does not apply. In
order to determine the groups $K_n(E,\cF)$, one can use localization:

\begin{ex}\label{ex:K-of-U_{2,3}}
Consider the matroid $U_{2,3}$, with $E=\{0,1,2\}$. As indicated in
Example \ref{ex:uniform}, its flats are $\emptyset$, the one-elements
subsets $\{0\}$, $\{1\}$ and $\{2\}$, and $\{0,1,2\}$. It follows that
the only proper flats are the singletons, and the Bergman fan of
$U_{2,3}$ is the fan in $\R^2$ with rays $\bu_1$, $\bu_2$, and $\bu_0
= -\bu_1 - \bu_2$, and has no two-dimensional cones. The associated toric
monoid scheme is the complement $V$ of the three closed points in
$\bbP^2$.  Since $V$ is smooth, $K_n(U_{2,3}) = \Kzar_n(V)\cong
K'_n(V)$.

To determine the groups $K'_n(V)$, we first note that the push-forward
$K'(\pt)\xrightarrow{i_*} K'(\bbP^2)$ does not depend on the closed
point. Theorem \ref{thm:K-loc} implies that there is a long exact
sequence
\begin{equation}\label{eq:U32-sequence}
	 \dotsc\to K'_n(\pt)^{\oplus 3}\xrightarrow{\oplus i_*} K'_n(\bbP^2)\to K'_n(V) \to K'_{n-1}(\pt)^{\oplus 3}\to \dotsc 
\end{equation}
It follows from the proof of Proposition \ref{prop:proj-space} that
$K'_n(\pt)\xrightarrow{i_*} K'_n(\bbP^2)$ is a composition of split
injective maps, and hence is split injective; therefore, the long
exact sequence \eqref{eq:U32-sequence} splits into short exact
sequences
\[
0\to K'_n(\bbP^2)/i_* \left(K'_n(\pt)\right) \to K'_n(V)\to K'_{n-1}(\pt)^{\oplus 2}\to 0.
\]
Applying Proposition \ref{prop:proj-space}, we can express
$K_n(U_{2,3}) = K'_n(V)$ as an
extension of stable homotopy groups of spheres:
\[
0\to (\pi_n^s)^{\oplus 2} \to K_n(U_{2,3})\to (\pi_{n-1}^s)^{\oplus 2}\to 0.
\] 
 \end{ex}

\vspace{1cm}
\appendix 
\section{Multiplicative structure of
  the $K$-theory of monoid schemes}\label{app:mult}

In Definition \ref{defn:K-Zar}, we defined a presheaf of spectra
$\Kzar$ on $\cMpctf$, satisfying Zariski descent. In this appendix, we
briefly describe how $\Kzar$ carries a natural structure of an
$E_\infty$-ring spectrum over which $K'$ is a module.

For convenience,
we will occasionally use the language of $(\infty,1)$-categories. In
particular, we will write ``limit'' for what would be called
``homotopy limit'' if we were working with model categories. We are
not aiming for the greatest possible generality; for example, the
hypotheses ``separated'' and ``finite type'' are not needed in
Proposition \ref{prop:Kzar-is-ring} below, but they make the
exposition easier.

Let $X$ be a pctf monoid scheme. Then the quasi-exact category (see
\cite[Section 2]{HW}) of finitely generated projective
$\mathcal{O}_X$-sets has a symmetric monoidal structure under smash
product, and this structure is bi-exact. Moreover, this category acts
via biexact smash product on the category of finitely generated pc
$\mathcal{O}_X$-sets.  Thus the following Lemma follows from Barwick
\cite[Cor.\,3.8.2]{Barwick},
where $K^Q(X)$ is the  spectrum defined in \ref{defn:K-Zar}.

\begin{lem}\label{lem:K_is_ring}
  Let $X$ be a pctf monoid scheme. The spectrum $\KQ(X)$
  has the natural structure of an $E_\infty$-ring spectrum,
  and $K'(X)$ is a module over $\KQ(X)$. 
\end{lem}

\begin{proof}
In particular, a quasi-exact category has the structure of a
Waldhausen category, where the cofibrations are admissible
monomorphisms and the weak equivalences are isomorphisms. It is
well-known (see \cite[1.9]{Waldhausen1126}) that the $K$-theory (via the
$wS_\bullet$-construction) of this Waldhausen category is equivalent
to the $K$-theory of the original quasi-exact category (via the
$Q$-construction). Now we can apply Barwick's result.
\end{proof}

\begin{lem}\label{lem:Kzar-as-finite-limit}
  Let $X$ be a separated pctf monoid scheme of finite type, and $\{U_i\}$
  a (finite) cover by affine open subschemes. Write $\mathcal{U}_\bullet$
  for the associated (\u{C}ech) simplicial scheme.
  Then we have a canonical weak homotopy equivalence
\[ \Kzar(X)\to \mathrm{lim}_{\Delta} \Kzar (\mathcal{U}_\bullet)
\simeq \mathrm{lim}_{\Delta} \KQ(\mathcal{U}_\bullet).
\]
\end{lem}

\begin{proof}
The first statement is a consequence of the fact that $\Kzar$ is a
sheaf in the Zariski topology, which is true by construction. Now the
second assertion follows from Lemma \ref{lem:affine-equivalence}.
\end{proof}

The forgetful functor, from the $\infty$-category of $E_\infty$-ring
spectra to the $\infty$-category of spectra, has a left adjoint (see
\cite{Harper}). It follows that the forgetful functor commutes with limits.

\begin{prop}\label{prop:Kzar-is-ring}
  Let $X$ be a separated pctf monoid scheme of finite type.
Then $\Kzar(X)$ has the natural structure of an $E_\infty$-ring spectrum. 
	
When $X$ is affine, the $E_\infty$ ring structure is transported from
$\KQ(X)$ via the weak homotopy equivalence $\KQ(X)\to
\Kzar(X)$. Moreover, $K'(X)$ is a module over $\Kzar(X)$.
\end{prop}

\begin{proof}
  As the forgetful functor to spectra commutes with limits,
  the first assertion follows from Lemmas \ref{lem:K_is_ring}
  and \ref{lem:Kzar-as-finite-limit}.

  It also follows from Lemma \ref{lem:K_is_ring} that
  $\bbH_\zar(X,K')$ is a module over $\Kzar(X)$. Since the
  presheaf $K'$ satisfies descent on $X_\zar$ by Lemma \ref{lem:K'-descent}, the second assertion follows. 
\end{proof}

Recall from Theorem \ref{thm:comparison} that $K'_0(X)\cong K'_0(X_k)$.

\begin{cor}\label{cor:vb-action}
  Let $X$ be a smooth toric monoid scheme of finite type, $k$ a field,
  and $n\geq 0$. Then $\Kzar_n(X)$ and $K'_n(X)$ have natural
  module structures over the commutative ring $\Kzar_0(X)\cong K_0(X_k)$.
  In particular, every vector bundle $\mathcal{E}$ on $X_k$
  determines a homomorphism $[\mathcal{E}]\cup - : K'_n(X) \to K'_n(X)$.
\end{cor}

\begin{proof}
	Combine Proposition \ref{prop:Kzar-is-ring}, Theorem \ref{thm:comparison} and Theorem \ref{thm:K-vs-K'}.\end{proof}

In our proof of the projective bundle formula in Section
\ref{sec:proj-bundle-formula}, we need a special case of the
projection formula in $K$-theory. Recall that $\{0,1\}$ is the initial
pointed commutative monoid. The category of finitely generated
$\{0,1\}$-sets is simply the category $\mathbf{Sets}_*$ of finite
pointed sets. Every short exact sequence of pointed sets splits,
meaning that if we are given a surjection $A\to B$ of pointed sets with kernel
$K$, there is a (unique) injection $j:B\to A$ such that $A = K\vee j(B)$.

Suppose now $X$ is a smooth toric monoid scheme, and $i:Y\to X$ is the
closed immersion of an equivariant smooth toric monoid
subscheme. Write $\Sh_Y$ and $\Sh_X$ for the categories of sheaves of
finitely generated pc $\mathcal{O}_Y$- and pc $\mathcal{O}_X$-sets,
respectively. Then we have
an exact pushforward functor $i_*:\Sh_Y\to \Sh_X$, and bi-exact action
functors \[\wedge_Y: \Sh_Y\times \mathbf{Sets}_*\to
\Sh_Y\;\mathrm{and}\;\wedge_X:\Sh_X\times \mathbf{Sets}_*\to \Sh_X\]
given by smash products.

Further, we write $q_Y:Y\to \pt$ and $q_X:X\to \pt$ for the structure
maps; here $\pt$ is the final monoid scheme $\MSpec\{0,1\}$.
\begin{lem}\label{lem:projection_formula_1}
	There is an identity of functors 
	\[\wedge_X\circ (i_*\times\mathrm{id}) = i_* \circ \wedge_Y:  \Sh_Y\times\mathbf{Sets}_*\to \Sh_X.\]
	 Therefore, if $y\in \Kzar_0 (Y)$, and $\alpha\in K'_n(\pt)$, then
	\[ i_*(y)\cup q_X^*(\alpha)=i_*(y\cup q_Y^*(\alpha)).\]
\end{lem}

\begin{proof}
The first identity follows by inspection. To conclude the second
identity, observe that, viewed as functions on pairs $(y,\alpha)$, the
left--hand side is induced by the functor $\wedge_X\circ
(i_*\times\mathrm{id})$, and the right--hand side by the functor $i_*
\circ \wedge_Y$.
\end{proof}

We conclude this appendix with a brief discussion of the action of the
symmetric monoidal category $\mathbf{Ab}_{\mathrm{free}}$ of free abelian groups on the stable
homotopy category. This is used in Section \ref{sec:K-of-sm-proj} to identify
$\Kzar(X)$ as a ring spectrum, for $X$ a smooth proper monoid
scheme. Given an additive category
$\mathfrak{C}$ admitting small coproducts (for example, the stable
homotopy category), there is a natural functor
\[
\mathbf{Ab}_{\mathrm{free}}\times \mathfrak{C}
\xrightarrow{\otimes} \mathfrak{C}
\]
 giving $\mathfrak{C}$ the structure of a category tensored over
 $\mathbf{Ab}_{\mathrm{free}}$. It is determined by the properties
 that $\Z\otimes -: \mathfrak{C}\to\mathfrak{C}$ is the identity
 functor, and $\otimes$ preserves coproducts in the first variable. If
 $\mathfrak{C}$ is triangulated (as is, for example, the stable
 homotopy category), then $\otimes$ is exact in the second variable.
 
Suppose that $(\mathfrak{C},\wedge)$ is moreover symmetric monoidal,
 and $\wedge$ preserves coproducts in each variable. If $R$ is a free
 abelian group, and $\mathbf{A}\in \mathfrak{C}$, there is a canonical
 equivalence
 \[
 (R\otimes\mathbf{A})\wedge (R\otimes\mathbf{A})\xrightarrow{\sim} (R\otimes_\Z R)\otimes (\mathbf{A}\wedge\mathbf{A})
 \]
in $\mathfrak{C}$. Moreover, if $R$ is a commutative ring with
multiplication $m:R\otimes_\Z R\to R$, and $\mathbf{A}$ is a
commutative ring in $\mathfrak{C}$ with multiplication
$\mu:\mathbf{A}\wedge\mathbf{A}\to\mathbf{A}$, then we have:
 
 \begin{prop}\label{prop:algebras}
 	 The map
 	\[ (R\otimes\mathbf{A})\wedge (R\otimes\mathbf{A})\xrightarrow{\sim} (R\otimes_\Z R)\otimes (\mathbf{A}\wedge\mathbf{A})\xrightarrow{m\otimes\mu} R\otimes\mathbf{A},\]
 	induced by the products on $R$ and $\mathbf{A},$
 	makes $R\otimes\mathbf{A}$ a commutative ring in $\mathfrak{C}$.
 \end{prop}
 
\begin{proof}
As explained in \cite[Sec.\,6]{Pa}, the functor $\otimes$ is
(strongly) monoidal in the first variable. It is now straightforward
(if tedious) to check that the operation given in the assertion endows
$R\otimes \mathbf{A}$ with the structure of a commutative ring in
$\mathfrak{C}$. In a much more general context this is proved, for
example, as part of \cite[Cor.\,3.6]{Sch}.
 \end{proof}
 
 In particular, suppose $\mathbf{E}$ is a commutative ring spectrum, and $R$ is a commutative ring that is free as abelian group. Then:

 \begin{cor}\label{cor:algebras}
$R\otimes\mathbf{E}$ has a natural structure of
commutative ring spectrum such that the following
natural map of graded rings is an isomorphism:
\[
R\oo_\Z \pi_* (\mathbf{E})   \xrightarrow{\sim}
\pi_*(R\otimes\mathbf{E})
\]
 \end{cor}

 \subsection*{Acknowledgements} The first author thanks
 Brooke Shipley and Lior Yanovski for pointing out intricacies
 around the monoidal structures discussed in the appendix. 


\subsection*{AI statement.} AI was used only for literature searches. 
 All mathematics and writing in this paper was done by humans.

\end{document}